\documentclass[11]{amsart}

\usepackage{amsfonts}
\usepackage{amscd}
\usepackage{amsmath, mathrsfs, amssymb, mathtools}
\usepackage{amsthm}
\usepackage{setspace}
\usepackage{hyperref}
\usepackage{color}
\usepackage{epsfig}
\usepackage{here}
\usepackage{graphicx}
\usepackage[all]{xy}
\usepackage{psfrag}
\usepackage{graphicx,transparent}
\usepackage{enumerate}
\usepackage{caption}

\theoremstyle{plain}
\newtheorem{theorem}{Theorem}[section]
\newtheorem{lemma}[theorem]{Lemma}
\newtheorem{proposition}[theorem]{Proposition}
\newtheorem{corollary}[theorem]{Corollary}

\theoremstyle{definition}

\newtheorem{definition}[theorem]{Definition}

\usepackage[
	    style=alphabetic,
            isbn=false,
            doi=false,
            url=false,
            backend=bibtex,
            maxnames = 5,
            sorting=nty,
           ]{biblatex}
\defbibheading{bibliography}[\bibname]{%
   \section*{References}%
}

\DeclareFieldFormat[article]{title}{{\it #1}}
\DeclareFieldFormat{journaltitle}{{\rm #1}}
\renewbibmacro{in:}{%
  \ifentrytype{article}{}{\printtext{\bibstring{in}\intitlepunct}}}
  \bibliography{biblio}

\title[Complete subvarieties in the strata of differentials]{Complete subvarieties in the projectivized strata of meromorphic differentials}

\author{Dawei Chen} 
\address[Dawei Chen]{Department of Mathematics, Boston College, Chestnut Hill, MA 02467, USA}
\email{dawei.chen@bc.edu}

\author{Guillaume Tahar}
\address[Guillaume Tahar]{Beijing Institute of Mathematical Sciences and Applications, Huairou District, Beijing, China}
\email{guillaume.tahar@bimsa.cn}
\date{\today}

\subjclass[2020]{Primary: 32G15; Secondary: 14H10, 14H15}

\keywords{Strata of meromorphic differentials, complete subvariety, saddle connection, polar domain, plurisubharmonic function}

\begin{document}

\begin{abstract}
Here we give an explicit construction of a globally defined strictly plurisubharmonic function on projectivized strata of strictly meromorphic differentials with prescribed orders of zeros and poles. In particular, this yields a flat-geometric proof that these strata do not contain positive-dimensional complete subvarieties.
\end{abstract}

\maketitle

\section{Introduction}

Let $\mu$ be a tuple of integers summing to $2g-2$. Denote by $\Omega\mathcal M_{g}(\mu)$ the stratum of Abelian differentials (i.e., differential one-forms) on smooth, connected, genus-$g$ complex projective curves whose zero and pole orders are prescribed by~$\mu$. Let $\mathbb P\Omega\mathcal{M}_g(\mu) = \Omega\mathcal{M}_g(\mu)/\mathbb C^{*}$ be the corresponding projectivized stratum, parameterizing canonical divisors associated to the differentials of type $\mu$. 

Abelian differentials induce translation structures on the underlying Riemann surfaces, with conical singularities at the zeros and poles. Affine transformations acting on them preserve the orders of these singularities. This viewpoint makes the study of strata of differentials an important topic in surface dynamics and moduli theory. We refer the reader to \cite{Zo06, Wr15, Ch17, CY26} for an introduction to this fascinating subject. 

Among the many aspects of strata of differentials, it is a natural and meaningful question to study how they appear from the perspective of affine geometry---for instance, whether a given stratum can contain a positive-dimensional complete algebraic subvariety. We remark that the analogous question of bounding the dimension of complete subvarieties in the moduli space of curves $\mathcal M_g$ has drawn considerable attention; see \cite{Di84, Lo95, Za99}. Despite these efforts, we still do not know whether $\mathcal M_4$ contains a complete algebraic surface. 

For a holomorphic signature~$\mu$, that is, when all entries of~$\mu$ are nonnegative, Gendron showed that the unprojectivized stratum $\Omega\mathcal M_{g}(\mu)$ of holomorphic differentials contains no positive-dimensional complete subvariety by applying the maximum modulus principle to shortest saddle connections; see~\cite{G20}. In this case, an alternative proof was later given by the first-named author, using the positivity of certain divisor classes on the moduli space of curves; see~\cite{C23}. 

When $\mu$ is a signature of strictly meromorphic differentials, again by exploiting the positivity of divisor classes, the first-named author proved that both the projectivized and unprojectivized strata, $\mathbb P\Omega\mathcal M_{g}(\mu)$ and $\Omega\mathcal M_{g}(\mu)$, of strictly meromorphic differentials contain no positive-dimensional complete subvarieties; see~\cite{C19, C24}. However, no alternative proof using flat geometry was known before. 

In this article, we provide a {\em flat-geometric} proof of the fact that the projectivized strata of strictly meromorphic differentials contain no positive-dimensional complete subvarieties. To this end, we combine a length function $\ell^{-2}$ introduced by Mondello in \cite{Mondello}, which measures the total length of a basis of short saddle connections, with a new function---the \textit{polar size function} $\mathcal{S}$---to construct a globally defined strictly plurisubharmonic function on these strata. 

\begin{definition}\label{defn:Mondello}
Given a translation surface $(X,\omega)$, let $B$ be a basis of the (relative) homology of $X$ punctured at the poles and taken relatively to the zeros of $\omega$, consisting of saddle connections. Define $$\ell^{-2}_{B}\coloneqq\sum\limits_{\gamma \in B} \left| P_{\gamma} \right|^{-2}$$
where $P_{\gamma}$ is the period of $\omega$ along an oriented saddle connection $\gamma$.
\par
On an unprojectivized stratum $\Omega\mathcal M_{g}(\mu)$, the function $\ell^{-2}$ is defined to be the supremum of the functions $\ell^{-2}_{B}$ over all bases $B$ formed by saddle connections.
\end{definition}

A tempting approach to obtaining a global function on the projectivized strata would be to normalize the function $\ell^{-2}$ by the area of the translation surface. However, not only do the flat structures induced by strictly meromorphic differentials have infinite area, but even for holomorphic differentials the area does not vary pluriharmonically with the period coordinates on the strata. This is why we introduce a new function instead. 
\par
In the flat structure induced by a strictly meromorphic differential, there exists a canonical family of saddle connections that decomposes the surface into the core and the polar domains (see Section~\ref{sub:CorePolarDomain} for details). We use these to construct the polar size function as follows. 

\begin{definition}\label{defn:PSF}
Let $(X,\omega)$ be a translation surface with $p \geq 1$ labeled poles. For each $i$, let $D_{i}$ be the set of boundary edges of the polar domain of the $i$th pole (oriented counterclockwise). Let $D_{\omega} = \bigsqcup\limits_{i=1}^{p} D_{i}$. In particular, the same saddle connection can appear twice in $D_{\omega}$, with possibly opposite orientations, if it bounds two polar domains (possibly the same). Elements of $D_{\omega}$ are \textit{oriented saddle connections}.
\par
The \textit{polar size function} is defined to be
$$
\mathcal{S}(\omega)\coloneqq \max\limits_{E \subset D_{\omega}} \left|\sum\limits_{\gamma \in E} P_{\gamma} \right|.
$$
\end{definition}

Combining $\ell^{-2}$ and $\mathcal{S}$, we obtain a function that is defined on the projectivized strata of meromorphic differentials. 

\begin{theorem}\label{theorem:MAIN}
For any projectivized stratum $\mathbb{P}\Omega\mathcal M_g(\mu)$ of strictly meromorphic differentials, the function $2\log \mathcal{S} + \log \ell^{-2}$ is strictly plurisubharmonic on $\mathbb{P}\Omega\mathcal M_g(\mu)$. 
\end{theorem}

From the above theorem, we deduce the existence of a non-constant plurisubharmonic function on any positive-dimensional subvariety of a projectivized stratum of meromorphic differentials. The maximum principle then implies that such subvarieties cannot be complete. 

\begin{corollary}\label{cor:NoPositive}
There is no positive-dimensional complete subvariety in any stratum $\mathbb{P}\Omega\mathcal M_g(\mu)$ of strictly meromorphic differentials. 
\end{corollary}

Besides Abelian differentials, one can also study strata of $k$-differentials with prescribed zero and pole orders, where a $k$-differential is a section of the $k$th power of the canonical bundle. Note that a $k$-differential with a pole of order at least $k$ (i.e., when the corresponding $(1/k)$-translation surface structure has infinite area) can be lifted via the canonical cyclic covering construction to the $k$th power of a strictly meromorphic differential one-form; see~\cite[Section 2]{BCGGM-k}. Corollary~\ref{cor:NoPositive} therefore implies the following.  

\begin{corollary}\label{cor:k-diff}
    For any stratum $\Omega^k\mathcal{M}_{g}(\mu)$ of $k$-differentials with a pole of order at least $k$, the projectivized stratum $\mathbb P\Omega^k\mathcal M_{g}(\mu)$ contains no positive-dimensional complete subvarieties.  
\end{corollary}

We remark that since $\mathbb{P}\Omega\mathcal{M}_{g}(\mu)$ is a quasi-projective variety, its complete algebraic subvarieties and complex analytic subvarieties coincide by the GAGA principle. Moreover, if a complete subvariety $M$ is singular, we may work with a resolution of singularities of $M$ and pull back the family of differentials accordingly. Alternatively, one can intersect a higher-dimensional complete subvariety with ample hypersurfaces until obtaining a complete algebraic curve. Therefore, Corollary~\ref{cor:NoPositive} is equivalent to showing that $\mathbb{P}\Omega\mathcal{M}_{g}(\mu)$ contains no complete algebraic curves, in which case we may work with the normalization of such a curve, which is smooth. In these senses, even when $M$ is singular, we can still speak of holomorphic local coordinates on $M$. We shall do so without further comment. 

Finally, we point out that the $\psi$ class at each pole (as an $S^1$-bundle modulo $\mathbb R^{+}$) has fiber represented by the boundary of the corresponding polar domain. Therefore, the polar size function we constructed and the related plurisubharmonic property provide an analytic incarnation of the positivity of $\psi$ classes utilized in the algebro-geometric proof of Corollary~\ref{cor:NoPositive} in \cite{C19, C24}. Note that this interplay between $\psi$ classes and the underlying metric geometry played an important role in Kontsevich's and Mirzakhani's proofs of Witten's conjecture, in the contexts of ribbon graphs associated to the flat metric of Strebel differentials and geodesic boundaries associated to the cusps of hyperbolic Riemann surfaces, respectively (see \cite{Ko92, Mi07}). In this sense, we expect that the ideas and techniques in this paper will inspire further developments in related fields. 

\subsection*{Acknowledgements} The research of D.C. was supported in part by the National Science Foundation under grant DMS-2301030 and by Simons Travel Support for Mathematicians. The research of G.T. was supported by the Beijing Natural Science Foundation (IS23005) and by the French National Research Agency through the project TIGerS (ANR-24-CE40-3604). The authors thank Quentin Gendron, Saul Schleimer, and Alex Wright for their interest, valuable remarks, and stimulating discussions. We also thank the referees for carefully reading the paper and providing very helpful comments.

\section{Period coordinates of strata of differentials}

Let $\mu = (m_1, \ldots, m_n)$ be a tuple of integers summing to $2g - 2$. Given $(X, \omega) \in \Omega\mathcal M_g(\mu)$, let $Z$ and $P$ be the sets of zeros and poles of $\omega$ in $X$, respectively. Integrating $\omega$ along the homology classes in the relative homology group $H_1(X\setminus P, Z; \mathbb Z)$ provides local coordinates on $\Omega\mathcal M_g(\mu)$ at $(X,\omega)$, called {\em period coordinates}.   

Note that $\omega$ induces a translation structure on $X \setminus (Z \cup P)$. A zero of order $m_i \ge 1$ corresponds, under the induced flat metric, to a {\em conical singularity} (also called a {\em saddle point}) with cone angle $2\pi (m_i + 1)$. A pole of order $m_j \leq -2$ has a flat-geometric neighborhood formed by gluing $|m_j|-1$ Euclidean planes at infinity. A simple pole of order $-1$ has a flat-geometric neighborhood given by a half-infinite cylinder. The period coordinates defined above describe local deformations of the translation surface structure while preserving the orders of its singularities, thus providing local coordinates for the stratum $\Omega\mathcal M_g(\mu)$. 

A {\em saddle connection} is a geodesic joining two zeros of $\omega$, and it is said to be closed if the two zeros coincide. Since one may choose the homology class representatives to be length-minimizing paths connecting the zeros of~$\omega$, it follows that saddle connections generate the relative homology group, or dually, that their periods generate the local coordinate system on $\Omega\mathcal M_g(\mu)$ near $(X, \omega)$. 

We refer to~\cite{Zo06} for a comprehensive survey of the strata of holomorphic differentials and the induced flat surface structures, and to \cite{Bo15} for an introduction to the strata of meromorphic differentials.

Recall that for a complex analytic space $\mathcal M$, an upper semi-continuous function 
$f\colon \mathcal M\to \mathbb R \cup \{ -\infty \}$ is said to be {\em plurisubharmonic} if for any holomorphic map $\varphi\colon \Delta \to \mathcal M$ the function $f\circ \varphi$ is subharmonic, where $\Delta\subset \mathbb C$ denotes the unit disk. Additionally, $f$ is \emph{strictly plurisubharmonic}
if for every $x\in\mathcal M$ there exist an open neighborhood $U\ni x$ and $\varepsilon>0$
such that for some (equivalently, any) holomorphic embedding $i\colon U\hookrightarrow \mathbb C^N$ the function $f - \varepsilon\cdot \|i(\cdot)\|^2$
is plurisubharmonic on $U$. Intuitively speaking, plurisubharmonicity means the function is ``convex'' in every complex direction. 

The length function $\ell^{-2}$, constructed in terms of lengths of saddle connections in Definition~\ref{defn:Mondello}, has good harmonicity properties. The following was shown by Mondello in \cite[Lemma 3.15]{Mondello}. Although Mondello stated it only for holomorphic differentials, the same proof works for strata of meromorphic differentials. 

\begin{proposition}\label{prop:Mondello}
For any unprojectivized stratum $\Omega\mathcal M_g(\mu)$ of meromorphic differentials, the function $\ell^{-2}$ satisfies the following properties:
\begin{itemize}
\item[(i)] $\ell^{-2}$ is homogeneous of degree $-2$;
\item[(ii)] $\ell^{-2}$ and $\log \ell^{-2}$ are plurisubharmonic on $\Omega\mathcal M_g(\mu)$, and strictly plurisubharmonic when restricted to a codimension-one submanifold transverse to the rays of $\Omega\mathcal M_g(\mu)$.
\end{itemize} 
\end{proposition}

In the above, the {\em homogeneous degree} of a function $f$ on a stratum of differentials being $d$ means 
$f(\lambda \omega) = |\lambda|^{d} f(\omega)$ for all $\lambda\in \mathbb C^{*}$ and all differentials $\omega$ parameterized in the stratum.  

\section{The polar size function}\label{sub:PolarSize}

\subsection{The core, polar domains, and peripheral saddle connections}\label{sub:CorePolarDomain}

In a translation structure induced by a meromorphic differential, every neighborhood of a pole has infinite area. However, the periods of such a differential are entirely captured by a subset, namely the \textit{core}, a finite-area domain cut out by saddle connections. 
\par 
\begin{definition}\label{defn:core}
A subset $E$ of a translation surface $(X,\omega)$ is \textit{convex} if and only if every geodesic segment joining two points of $E$ lies entirely in~$E$.
\par
The \textit{convex hull} of a subset $F$ of a translation surface $(X,\omega)$ is the smallest closed convex subset of $X$ that contains~$F$.
\par
We define the {\em core} of $(X,\omega)$ to be the convex hull of the zeros of $\omega$ in $(X,\omega)$, denoted $\text{core}(X,\omega)$.  
\end{definition}
\par 
The core separates the poles from one another. The following result shows that the complement of the core has as many connected components as the number of poles (see Proposition~4.4 and Lemma~4.5 of \cite{Ta18}).    
\par 
\begin{proposition}\label{prop:decomppoles}
For a translation surface $(X,\omega)$, the boundary of the core is a finite union of saddle connections. Moreover, each connected component of $X \setminus \text{core}(X,\omega)$ is a topological disk that contains a unique pole. 
\end{proposition}

We refer to these connected components as \textit{polar domains}. A saddle connection of $(X,\omega)$ that belongs to the boundary of some polar domain is called a \textit{peripheral saddle connection}. Along a deformation inside a stratum, the appearance or disappearance of a peripheral saddle connection occurs when one crosses the real hypersurface called the {\em discriminant} of the stratum (see Propositions~4.11, 4.14, and 4.15 of \cite{Ta18}). 

\begin{proposition}\label{prop:Discriminant}
A translation surface $(X,\omega)$ belongs to the discriminant $\Delta$ of the stratum if and only if two consecutive peripheral saddle connections with independent relative homology classes meet at an angle of $\pi$.
\par
On the complement of $\Delta$, the decomposition of the surface into the core and polar domains is structurally stable. 
\end{proposition}

\subsection{Peripheral homology classes}

We recall the following notation introduced in Definition~\ref{defn:PSF}. In a translation surface $(X,\omega)$ with $p$ labeled poles, for $1 \leq i \leq p$, let $D_{i}$ be the set of boundary edges of the polar domain of the $i$th pole (oriented counterclockwise), and let $D_{\omega} = \bigsqcup\limits_{i=1}^{p} D_{i}$. Elements of $D_{\omega}$ are oriented peripheral saddle connections.

\begin{definition}
Given a translation surface $(X,\omega)$, a \textit{peripheral homology class} is a class $[\delta]$ of the form $\sum\limits_{\gamma \in E} [\gamma]$ for some $E \subset D_\omega$. We denote by $PH_{\omega}$ the (finite) set of peripheral homology classes of $(X,\omega)$. 
\end{definition}

\begin{lemma}\label{lem:PeripheralDeform}
Given a translation surface $(X,\omega)$ in a stratum $\Omega\mathcal{M}_{g}(\mu)$ of meromorphic differentials, there exists a contractible neighborhood $V$ of $(X,\omega)$ 
such that:
\begin{itemize}
\item the peripheral homology classes in $PH_{\omega}$ are marked throughout $V$;
\item for any $\omega' \in V$, every homology class in $PH_{\omega'}$ is a deformation of a peripheral homology class in $PH_{\omega}$.
\end{itemize} 
\end{lemma}

\begin{proof}
For a sufficiently small neighborhood $V$ of $(X,\omega)$, 
every saddle connection of $D_{\omega}$ persists as a saddle connection when $(X,\omega)$ deforms in $V$. Observe that peripheral saddle connections can fail to remain peripheral only when deforming (concavely) two consecutive peripheral saddle connections that are in the same direction. In this case, the sum of the two original peripheral homology classes gives the homology class of the new peripheral saddle connection. 
\end{proof}

\subsection{Coherent homology classes}

Some of the homology classes defined below have remarkable deformation properties, which will be crucial in our proof that the polar size function introduced in Definition~\ref{defn:PSF} is plurisubharmonic. 

\begin{definition}
Given a translation surface $(X,\omega)$ in a stratum of meromorphic differentials, a \textit{coherent homology class} is a class $[\delta]$ of the form $\sum\limits_{\gamma \in E} [\gamma]$ for some $E \subset D_\omega$ such that there exists $\theta \in \mathbb{R}$ for which $\Re\left(e^{-i\theta} P_{\gamma} \right)>0$ for all $\gamma \in E$ and $\Re\left(e^{-i\theta} P_{\gamma} \right)<0$ for all $\gamma \in D_\omega\setminus E$.
\par
We denote by $Coh_{\omega}$ the (finite) set of coherent homology classes in $(X,\omega)$. 
\end{definition}

Since $D_\omega$ is a finite set, there are only finitely many coherent homology classes in $(X,\omega)$. Additionally, coherent homology classes arise from choosing an angle $\theta$ that differs from the arguments of all $\gamma$ in $D_\omega$. Therefore, the set $Coh_{\omega}$ of coherent homology classes in $(X,\omega)$ is nonempty.   

\begin{lemma}\label{lem:Persistence}
Given a translation surface $(X,\omega)$ in a stratum $\Omega\mathcal{M}_{g}(\mu)$ of meromorphic differentials, every coherent homology class of $(X,\omega)$ persists as a coherent homology class in a neighborhood of $(X,\omega)$ in $\Omega\mathcal{M}_{g}(\mu)$.
\end{lemma}

\begin{proof}
Let $[\delta]=\sum\limits_{\gamma \in E} [\gamma]$ be a coherent homology class of $(X, \omega)$. In particular, there exists $\theta \in \mathbb{R}$ such that for any class $[\gamma] \in E$, $\Re\left(e^{-i\theta} P_{\gamma} \right)>0$.
\par
In the boundary of any polar domain, along an arbitrarily small deformation, a peripheral saddle connection $\gamma$ representing $[\gamma]$ whose incident corners have angle strictly larger than $\pi$ persists as a peripheral saddle connection. Moreover, $\Re\left(e^{-i\theta} P_{\gamma} \right)$ stays positive.
\par
In contrast, given some polar domain, consider a pair of consecutive boundary oriented saddle connections $\gamma_{1},\gamma_{2}$ such that:
\begin{itemize}
    \item $\gamma_{1}$ and $\gamma_{2}$ meet at a boundary corner of angle $\pi$;
    \item the starting point of $\gamma_{1}$ and the ending point of $\gamma_{2}$ are corners of angle strictly larger than $\pi$.
\end{itemize}
We have that $[\gamma_{1}]$ and $[\gamma_{2}]$ belong to $D_\omega$. Assuming that they belong to $E$, $\Re\left(e^{-i\theta} P_{\gamma_{1}} \right)$ and $\Re\left(e^{-i\theta} P_{\gamma_{2}} \right)$ stay positive for any small enough deformation. If along an arbitrarily small deformation, the angle at the corner between $\gamma_{1}$ and $\gamma_{2}$ becomes strictly smaller than $\pi$, then $\gamma_{1}$ and $\gamma_{2}$ cease to be peripheral saddle connections for this polar domain. However, $\gamma_{1}$ and $\gamma_{2}$ form a triangle with a third saddle connection $\gamma'$ satisfying $[\gamma']=[\gamma_{1}]+[\gamma_{2}]$ (the assumption that boundary edges of the same polar domain have counterclockwise orientation is crucial here). This new saddle connection $\gamma'$ is a peripheral saddle connection for the same polar domain and $\Re\left(e^{-i\theta} P_{\gamma'} \right)= \Re\left(e^{-i\theta} P_{\gamma_{1}} \right) +\Re\left(e^{-i\theta} P_{\gamma_{2}} \right)$ stays positive provided that the deformation is small enough. In this case, $[\gamma_{1}]$ and $[\gamma_{2}]$ disappear from $E$ but they are replaced by $[\gamma']$ so that the homology class $[\delta]=\sum\limits_{\gamma \in E} [\gamma]$  does not change.
\par
More generally, consider a chain of consecutive boundary oriented saddle connections $\gamma_{1},\dots,\gamma_{k}$ meeting at corners of angle $\pi$ and such that $[\gamma_{1}],\dots,[\gamma_{k}]$ belong to $E$. Under an arbitrarily small deformation, some of these saddle connections may cease to be peripheral. However, they are replaced by newly created peripheral saddle connections in such a way that the homology class $[\delta]=\sum\limits_{\gamma \in E} [\gamma]$ remains unchanged, see Figure~\ref{Figure:Deformation}.
\par
In the specific case where the chain of consecutive saddle connections is cyclic, it forms the boundary of a polar domain consisting of a semi-infinite translation cylinder foliated by a family of parallel closed geodesics. Denoting by $\alpha_{1}(\varepsilon),\dots,\alpha_{k}(\varepsilon)$ the sequence of corner angles along a deformation parametrized by $\varepsilon$, the total angle defect
$
\sum\limits_{i=1}^{k}\bigl(\pi-\alpha_i(\varepsilon)\bigr)
$
can be interpreted as the winding number of a loop homotopic to a waist curve of the cylinder, and is therefore identically equal to $0$ throughout the deformation. Consequently, the corner angles cannot all become strictly smaller than $\pi$, so the previous arguments still apply, and the boundary of the polar domain continues to be formed by peripheral saddle connections whose homology classes belong to $E$.
 In all cases,
$
[\delta]=\sum_{\gamma\in E}[\gamma]
$
persists as a coherent homology class under sufficiently small deformations.
\end{proof}

\begin{figure}
\includegraphics[scale=0.3]{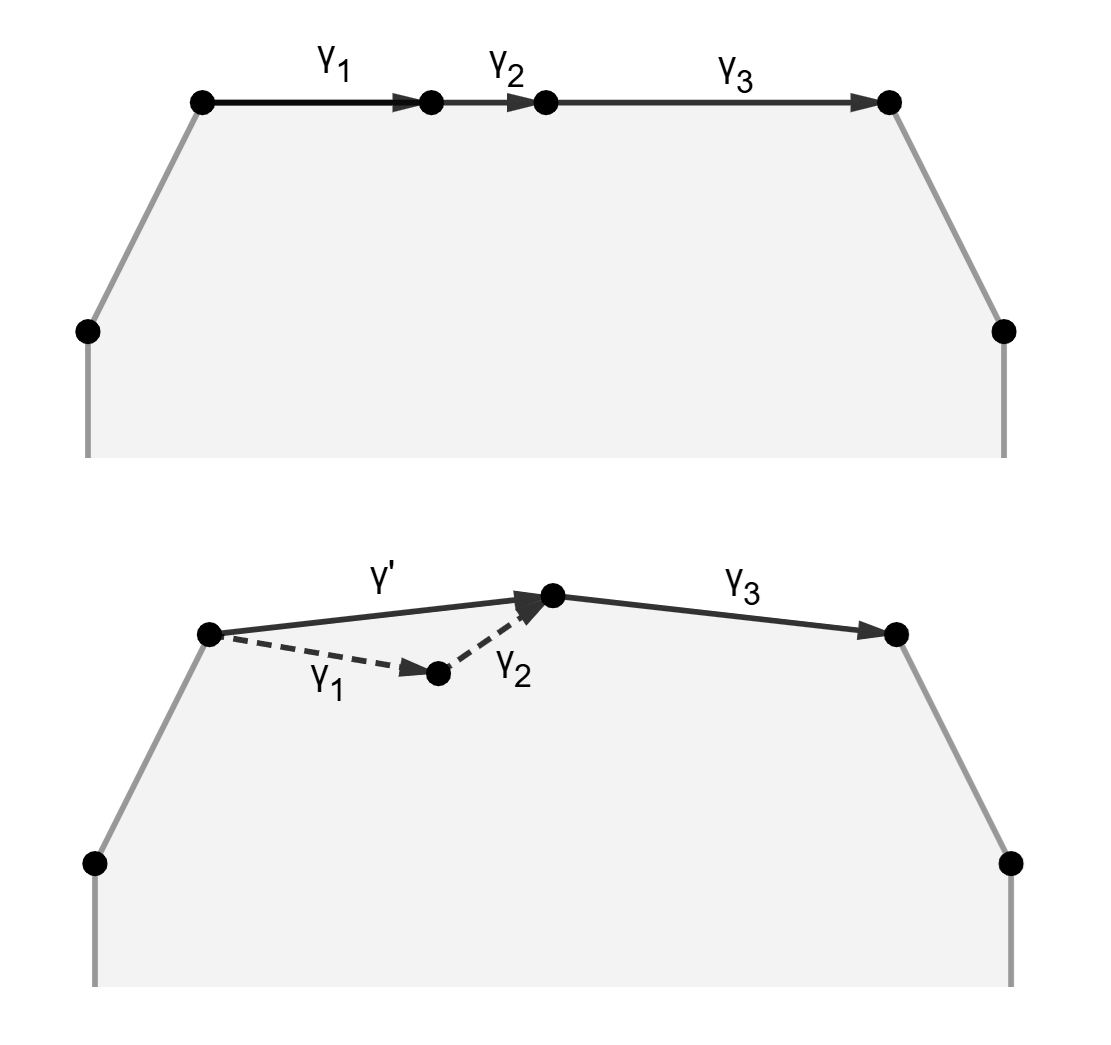}
\caption{Deformation of a chain of peripheral saddle connections}\label{Figure:Deformation}
\end{figure}

\subsection{Properties of the polar size function}

To show that the polar size function $\mathcal{S}$ introduced in Definition~\ref{defn:PSF} is plurisubharmonic, the key observation is that only coherent classes appear in the maximization. 

\begin{lemma}\label{lem:NotMaximizer}
Given a translation surface $(X,\omega)$ in a stratum of meromorphic differentials, for any subset $E \subset D_{\omega}$, if $\left|\sum\limits_{\gamma \in E} P_{\gamma} \right|= \mathcal{S}(\omega)$, then $\sum\limits_{\gamma \in E} [\gamma]$ is coherent. 
\end{lemma}

\begin{proof}
Since $D_\omega$ is nonempty, we have $\mathcal{S}(\omega)>0$. Given a subset $E$ such that $\left|\sum\limits_{\gamma \in E} P_{\gamma} \right|= \mathcal{S}(\omega)$, we denote by $\theta$ the argument of $\sum\limits_{\gamma \in E} P_{\gamma}$. Then $\mathcal{S}(\omega)=\Re\left(e^{-i\theta}\sum\limits_{\gamma \in E} P_{\gamma} \right)$. Up to replacing $\omega$ by $e^{-i\theta}\omega$, we assume that $\sum\limits_{\gamma \in E} P_{\gamma}$ is a positive real number and coincides with $\mathcal{S}(\omega)$.

If there is an element $\gamma$ of $E$ such that $\Re(P_{\gamma})<0$, we can remove it from $E$ to increase $\Re\left(\sum\limits_{\gamma \in E} P_{\gamma} \right)$ and obtain a modulus bigger than $\mathcal{S}(\omega)$, which leads to a contradiction. For the same reason, there cannot be any element $\gamma$ of $D_{\omega} \setminus E$ such that $\Re(P_{\gamma})>0$. Therefore, $\sum\limits_{\gamma \in E} [\gamma]$ can fail to be coherent only if there exists a nonempty subset $F \subset D_{\omega}$ consisting of those $\gamma$ for which $\Re\left(P_{\gamma}\right)=0$. 
Since the elements in $F$ are purely imaginary, adding or removing an element of $F$ to $E$ yields a sum of periods with larger modulus than $\mathcal{S}(\omega)$, again a contradiction. Thus the homology class $\sum\limits_{\gamma \in E} [\gamma]$ is coherent. 
\end{proof}

\begin{proposition}\label{prop:PSFLocal}
Given a translation surface $(X,\omega)$ in a stratum $\Omega\mathcal M_g(\mu)$ of meromorphic differentials, there exists a neighborhood $V$ of $(X,\omega)$ such that, for all $(X', \omega')$ in $V$, the function $\mathcal{S}(\omega')$ is given by the maximum of functions of the form $\left|\sum\limits_{\gamma \in E} P_{\gamma} \right|$, where $\sum\limits_{\gamma \in E} [\gamma]$ is a coherent homology class of $(X,\omega)$ that extends to the entire $V$. 
\end{proposition}

\begin{proof}
We take a sufficiently small contractible neighborhood $V$ of $(X,\omega)$ such that the markings of the (finitely many) peripheral homology classes of $PH_{\omega}$ extend consistently in $V$ (although some of these homology classes may fail to stay peripheral when $(X,\omega)$ deforms in $V$, see Lemma~\ref{lem:PeripheralDeform}). Up to shrinking $V$, we can assume that every coherent class $\delta \in Coh_{\omega}$ extends to $V$ as a coherent class (Lemma~\ref{lem:Persistence}). Moreover, by finiteness of $PH_{\omega}$, there exists an $\epsilon>0$ such that for every class $\delta' \in PH_{\omega} \setminus Coh_{\omega}$, $|\int_{\delta'} \omega'|<\mathcal{S}(\omega)-\epsilon$ (Lemma~\ref{lem:NotMaximizer}).
\par
Let $\delta \in Coh_{\omega}$ such that $|\int_{\delta} \omega|=\mathcal{S}(\omega)$. Up to further shrinking $V$, we can assume that for all $(X',\omega')\in V$, $|\int_{\delta} \omega'|>\mathcal{S}(\omega)-\epsilon$, and therefore, $\mathcal{S}(\omega')$ is realized by a coherent homology class in $Coh_{\omega}$ for all $(X',\omega') \in V$. 
\end{proof}

We can thus deduce that $\log \mathcal{S}(\omega)$ is plurisubharmonic. 

\begin{proposition}\label{lem:functionS}
The function $\mathcal{S}$ has the following properties:
\begin{itemize}
\item[(i)] $\mathcal{S}$ is homogeneous of degree $1$ on $\Omega\mathcal M_g(\mu)$;
\item[(ii)] $\log \mathcal{S}$ is plurisubharmonic on $\Omega\mathcal M_g(\mu)$.
\end{itemize} 
\end{proposition}

\begin{proof}
Claim (i) immediately follows from the definition. Claim (ii) follows from Proposition~\ref{prop:PSFLocal}, since the period of a (coherent) homology class is holomorphic, $\log |\cdot|$ is plurisubharmonic, and the pointwise maximum of finitely many plurisubharmonic functions remains plurisubharmonic. 
\end{proof}

\section{Proof of the main result}\label{sec:Main}

Combining Proposition~\ref{prop:Mondello} and the results of Section~\ref{sub:PolarSize}, we will show that $2\log \mathcal{S} + \log \ell^{-2}$ is a strictly plurisubharmonic function on any projectivized stratum $\mathbb{P}\Omega\mathcal M_g(\mu)$ of meromorphic differentials. 

\begin{proof}[Proof of Theorem~\ref{theorem:MAIN}]
First observe that the degrees of homogeneity of $\mathcal{S}$ and $\ell^{-2}$ are respectively $1$ and $-2$. Therefore, $2\log \mathcal{S} + \log \ell^{-2}$ is well-defined on $\mathbb{P}\Omega\mathcal M_g(\mu)$.
\par
It remains to prove that $2\log \mathcal{S} + \log \ell^{-2}$ is strictly plurisubharmonic. This follows from the following two statements:
\begin{itemize}
\item $\log \ell^{-2}$ is strictly plurisubharmonic when restricted to a codimension-one submanifold transverse to the rays of $\Omega\mathcal M_g(\mu)$ (see Proposition~\ref{prop:Mondello});
\item $\log \mathcal{S}$ is plurisubharmonic on $\Omega\mathcal M_g(\mu)$ (see Proposition~\ref{lem:functionS}).
\end{itemize} 
\end{proof}

Finally, we conclude that there is no positive-dimensional complete subvariety in a projectivized stratum $\mathbb{P}\Omega\mathcal{M}_{g}(\mu)$ of strictly meromorphic differentials. 

\begin{proof}[Proof of Corollary~\ref{cor:NoPositive}]
Since the function $2\log \mathcal{S} + \log \ell^{-2}$ is strictly plurisubharmonic, its restriction to any positive-dimensional complete subvariety of $\mathbb{P}\Omega\mathcal{M}_{g}(\mu)$ is nonconstant. However, by the maximum principle, there is no nonconstant plurisubharmonic function on a positive-dimensional complete variety. We thus obtain the desired claim. 
\end{proof}

\printbibliography
\end{document}